\documentclass[11pt]{article}
\usepackage{makeidx}
\usepackage[centertags]{amsmath}
\usepackage{amsthm}
\usepackage{newlfont}
\usepackage{amsmath}
\usepackage{amsfonts,amsmath,latexsym}
\usepackage{mathrsfs}
\usepackage{amssymb}
\usepackage{amsbsy}
\usepackage{indentfirst}
\usepackage{mathrsfs}
\usepackage{graphicx}
\usepackage{epstopdf}
\usepackage{subfigure}
\usepackage{caption}
\usepackage{grffile}
\usepackage{float}
\usepackage{amssymb, amsmath}
\usepackage{xcolor}
\usepackage{enumitem}
\usepackage{CJK, CJKnumb, CJKulem, verbatim}
\usepackage[colorlinks, citecolor=blue, linkcolor=blue]{hyperref}
\usepackage{diagbox}
\usepackage{authblk} 

\titlepage
\newlength{\defbaselineskip}
\newcommand{\ep}{\epsilon}

\newcommand{\be}{\begin{eqnarray}}
\newcommand{\ee}{\end{eqnarray}}
\newcommand{\bestar}{\begin{eqnarray*}}
\newcommand{\eestar}{\end{eqnarray*}}

\newcommand{\nn}{\nonumber}
\newcommand{\ignore}[1]{}
{} \theoremstyle{plain}
\newtheorem{thm}{Theorem}[section]

\newtheorem{lemma}{Lemma}[section]

\newtheorem{exam}{Example}[section]
\theoremstyle{definition}

\newtheorem{rem}{Remark}[section]
\numberwithin{equation}{section}

\def\k{\kappa}

\def\>{\geq}
\def\<{\leq}

\allowdisplaybreaks
\graphicspath{{./graphics/}}

\title{A central limit theorem  for unbalanced step-reinforced random walks}

\author[a]{Zhishui Hu}
\author[b]{Liang Dong\thanks{ E-mail addresses: huzs@ustc.edu.cn (Z. Hu), 749287661@qq.com (L. Dong)}}

\affil[a]{Department of Statistics and Finance, School of Management\par
	University of Science and Technology of China\par
	Hefei, Anhui 230026, China}
\affil[b]{School of Mathematics and Statistics,
	Suzhou University of Technology\par
	Changshu, Jiangsu 215500, China}

\date{}

\begin{document}
	\maketitle
\begin{abstract} 

In this paper, we study a class of unbalanced step-reinforced random walks that unifies   the  elephant random
 walk,  the positively step-reinforced random 
walk, and the negatively step-reinforced random 
walk. By establishing a connection with  bond percolation on  random recursive trees,   these processes can be represented as randomly weighted sums
of independent and identically distributed random variables. We first derive normal and stable central limit theorems for such randomly weighted sums,  and then apply  these results
to obtain a unified central limit theorem for   unbalanced step-reinforced random 
walks.

\vskip 0.2cm \noindent{\it Key words:}  Reinforcement,  random walk, stable limit theorem, randomly weighted sum
 \vskip 0.2cm

\noindent{\it Mathematics Subject Classification}
 60F05;  60G50;   
 60G52
   
\end{abstract}

\section{Introduction and main result}  \label{sect1}

Let $\xi_1, \xi_2, \cdots$ be a sequence of i.i.d. random variables. Aguech et al. \cite{A2025} introduced a new class of step-reinforced random walk,  defined as follows: let $p$ and $r$ be two fixed parameters in $[0,1]$, set  $X_1=\xi_1$ and for $n\ge 2$ define recursively
\bestar
{X}_n :=
\left\{
\begin{array}{ccc}
{X}_{U_n}, & \text{with probability} & rp,\\
 -X_{U_n}, & \text{with probability} & (1-r)p, \\
 \xi_{n}, & \text{with probability} & 1-p, \\
\end{array}
\right.
\eestar
where $\{U_n, n\ge 2\}$ is a sequence of  independent random variables such that for each $n$, $U_n$ is uniformly distributed on $ \{ 1,2, \cdots ,n-1 \}$, and the sequences $\{U_n\}$ and $\{\xi_k\}$ are independent.  The   process 
\bestar
T_n=\sum_{k=1}^n X_k,~~n\ge 1
\eestar
 is called the unbalanced step-reinforced random walk.

A simple and well-known  example of a step-reinforced random walk  is the elephant random
 walk (ERW),   introduced by Sch\"utz and Trimper \cite{ST2004}. The ERW has attracted considerable attention in  recent years; see 
\cite{BB2016,  BE2017, B2022, Col2017a, Col2017b, DFHM2023, KT2019, K2016, Qin2025} and the references therein. As generalizations of the ERW, the positively and negatively step-reinforced random walks were introduced  by Simon \cite{S1955} and Bertoin \cite{B2024}, respectively,  and have been extensively studied  in \cite{BR2022, B2021A, B2021B, B2024, Bu2018, HZ2024, H2025}. The unbalanced step-reinforced random walk  provides a unified framework that includes the ERW and the positively and negatively step-reinforced random walks.  Specifically,  when $p=1$ and $\mathbb{P}(\xi_1=1)=1-\mathbb{P}(\xi_1=-1)=s$ with $s\in [0,1]$,
the process  $(T_n)_{n\ge 1}$ reduces to the ERW; when  $r=1$, $(T_n)_{n\ge 1}$ corresponds to the   positively  step-reinforced random walk; and when $r=0$,  it corresponds to the  negatively  step-reinforced random walk.  

 Aguech et al. \cite{A2025} established  strong laws of large numbers and central limit theorems for the unbalanced step-reinforced random walk  $(T_n)_{n\ge 1}$.
Under the moment condition $\mathbb{E}(\xi_1^2)<\infty$,  they proved that  a phase transition occurs at the critical point $a:=(2r-1)p=1/2$.

\begin{thm}\label{th1.1} (\cite{A2025}) Let $\xi_1, \xi_2, \cdots$ are i.i.d. random variables with finite second moment. Define
\bestar
a:=(2r-1)p, \qquad \mu:=\frac{1-r}{1-a}\mathbb{E}(\xi_1), \qquad \sigma^2:=\mathbb{E}(\xi_1^2)-\mu^2.
\eestar
Then the following limit theorems hold:
\begin{enumerate}
\item[(i)] Subcritical regime ($a<1/2$):  
\be 
\frac{T_n-n\mu}{\sqrt{n}} \stackrel{d}{\longrightarrow} N(0, \sigma^2/(1-2a)); \label{aresult}
\ee 

\item[(ii)] Critical regime ($a=1/2$):
\bestar
\frac{T_n-n\mu}{\sqrt{n \log n}} \stackrel{d}{\longrightarrow} N(0, \sigma^2);
\eestar

\item[(iii)]Supercritical regime ($a>1/2$):
\bestar
\frac{T_n-n\mu}{n^{a}}  \longrightarrow  L~~\mbox{a.s.},
\eestar
where $L$ is non-gaussian square-integrable random variable.
\end{enumerate}
\end{thm}

As shown in Theorem \ref{th1.1}, when the second moment is finite, $T_n$ 
converges in distribution to a normal law in both the subcritical and critical regimes. 
A natural question is whether analogous limit theorems hold when $\xi_1$ belongs to the domain of attraction of a normal distribution. 
To the best of our knowledge,  no such results have been established even for the positively or negatively step-reinforced random walks (i.e., $r=0$ or $1$). Furthermore, 
the case in which $\xi_1$ belongs to the domain of attraction of a stable distribution is also of interest.

In this paper, we focus on the subcritical regime. For simplicity, we restrict our attention to the case where $\xi_1$ belongs to the domain of normal attraction of either a normal distribution or a symmetric stable distribution.
To state our main result, we first define a special ERW.
Set $X_1^0=1$ and for $n\ge 2$,
\bestar
{X}_n^0:=
\left\{
\begin{array}{ll}
{X}^0_{U_n}, & \text{with probability}~ r,\\
 -X^0_{U_n}, & \text{with probability}~ 1-r,
\end{array}
\right.
\eestar
where $\{U_n, n\ge 2\}$ is a sequence of  independent random variables with each $U_n$
uniformly distributed on $\{1, 2,\cdots, n-1\}$.
Then
$
T_n^0:=\sum_{k=1}^n X^0_n
$
is an ERW with  initial value  $T_1^0=1$ and the memory parameter $r$.

We now present the main result, which provides a unified central limit theorem for the unbalanced step-reinforced random walk.

\begin{thm} \label{th1027}
	Assume that $\alpha \in (0,2],  p\in (0,1), r\in [0,1]$ and $(2r-1)\alpha p<1$.
	Let $\xi_1, \xi_2, \cdots$ be i.i.d. random variables such that
\bestar
	\frac{1}{a_n}\sum_{k=1}^n \xi_k  \stackrel{d}{\rightarrow} S,
\eestar
 where $S$ is a symmetric $\alpha$-stable random variable 
with characteristic function 
$
\mathbb{E}(e^{it S})=e^{-|t|^{\alpha}}.
$
Then
\bestar
		\frac{T_n}{a_n} \stackrel{d}{\longrightarrow} (c(\alpha,p,r))^{1/\alpha} S,
		\eestar
where
\be 
c(\alpha,p,r)=\frac{1-p}{p}\sum_{k=1}^{\infty} \mathbb{E}(|T_k^0|^{\alpha})\mathrm{B}(k, 1+1/p)     \label{calphapr}
	\ee 
and $B$ denotes the beta function.
\end{thm}

\begin{rem} \label{rema1}
When $\alpha=2$ and $(4r-2)p<1$, we have 
\be
c(2, p,r)=\frac{1}{1-(4r-2)p}   \label{c2pr}
\ee
(see  Section \ref{sect4} for the proof).
Consequently, Theorem \ref{th1027} implies that if $\sum_{k=1}^n \xi_k/a_n  \stackrel{d}{\rightarrow} N(0,1)$ and $(4r-2)p<1$, then 
\bestar
\frac{T_n}{a_n} \stackrel{d}{\longrightarrow}  N\Big(0, \frac{1}{1-(4r-2)p}\Big).
\eestar
This generalizes the  limit theorem (\ref{aresult}) obtained by  Aguech et al. \cite{A2025}.
\end{rem}

\begin{rem} Assume that $\alpha\in (0, 2)$.
For the positively step-reinforced random walk  in the subcritical regime (i.e., $r=1$ and $\alpha p<1$),
Businger \cite{Bu2018} obtained that if $\xi_1$ has the same distribution as $S$,
then $n^{-1/\alpha}T_n\stackrel{d}{\rightarrow} (c(\alpha, p))^{1/\alpha} S$, where
$
c(\alpha,p)=\int_0^{1-p} f(x)dx
$
with
\bestar
f(x)=\sum_{k=1}^{\infty}k^{\alpha} \Big(1-\Big(\frac{x}{1-p}\Big)^p\Big)^{k-1} \Big(\frac{x}{1-p}\Big)^p,
~~~~x\in [0,1-p).
\eestar 
Note that in this case,  $T_k^0=k$, and a direct calculation  shows that
\bestar
c(\alpha,p,1)=\frac{1-p}{p}\sum_{k=1}^{\infty} \k^{\alpha}\mathrm{B}(k, 1+1/p) =\int_0^{1-p} f(x)dx.
\eestar
Thus,  the constant $c(\alpha, p)$  coincides with $c(\alpha,p,1)$.
Bertoin \cite{B2021B} extended  Businger's result to the case that $\xi_1 $
follows a general stable distribution.
Bertoin \cite{B2024} further derived   corresponding results for  the negatively step-reinforced random walk, where the characteristic function of the limiting distribution is expressed in a complex form. 

\end{rem}

We now outline the proof strategy for   our main result Theorem \ref{th1027}. The proof consists of three main steps:
(i) representing $T_n$ as a randomly weighted sum of the sequence $\{\xi_k\}$ by using an approach similar to \cite{B2024};
(ii) establishing normal and stable central limit theorems for general randomly weighted sums;
and (iii) proving the limit theorem for $T_n$ by applying these general results.

The remainder of this paper is organized as follows.  Section \ref{Sect2}  introduces three central limit theorems (Theorems \ref{th1}--\ref{th3}) for general randomly weighted sums. Their
proofs  are provided in Section \ref{sec3}  by using the characteristic function  method. Finally,  in
Section \ref{sect4}, we present the proof of  Theorem \ref{th1027}.

\section{Central limit theorems for randomly weighted sums}\label{Sect2}

Let $\{\xi_k , k\ge 1\}$ be a sequence of i.i.d. random variables, and let $\{W_{nk}, n\ge 1, k=1,\cdots, n\}$ be an array of random weights   independent of $\{\xi_k \}$.
Randomly weighted sums of the form 
$
\sum_{k=1}^n W_{nk}\xi_k 
$
play an important role in probability and statistics.  Below are some examples of such sums.  

The bootstrap sample  mean, 
introduced by Efron \cite{E1979},  is defined as  
$
\sum_{k=1}^n M_{nk}\xi_k ,
$
where  $(M_{n1},\cdots,  M_{nn})$ follows  a multinomial distribution with $n$ trials and success probability $1/n$ per cell, and is independent of $\{\xi_k \}$.

 A Bayesian analogue of the bootstrap was defined by
Rudin \cite{R1981}. It is constructed by letting $R_1, \cdots, R_{n-1}$ be i.i.d. uniform $(0,1)$ random variables independent of $\{\xi_k \}$. Denoting their order statistics by $R_{n1} \le \cdots \le R_{n,n-1} $ 
and setting $R_{n0}=0$ and $R_{nn}=1$,  the gaps $D_{nk}=R_{nk}-R_{n, k-1}$ for $ 1\le k\le n$ form the weights, yielding the Bayesian bootstrap mean
$\sum_{k=1}^n D_{nk}\xi_k .$ 

Another example is the self-normalized sum
 studied by  Breiman \cite{Br1965}, which is given by 
$ \sum_{k=1}^n \xi_k Y_k/\sum_{k=1}^n Y_k
$, where  we define $0/0=0$ 
and  $\{Y_k\}$ is a sequence of  i.i.d. non-negative, non-degenerate random variables independent of $\{\xi_k \}$.

A key characteristic  in all the examples above is that  the weights
$\{W_{nk},  k=1,\cdots, n\}$   are exchangeable for any fixed $n\ge 1$. 
However,  applications often involve sums with non-exchangeable weights.
The main object
of this section is to establish central limit theorems, both normal and stable, for general randomly weighted sums.
We begin with a simple result obtained by applying the Lindeberg  central limit theorem.

\begin{thm} \label{th1} Assume that $\mathbb{E}(\xi_1)=0$ and $\mbox{Var}(\xi_1)=1$, and that the weights satisfy the following conditions:
\begin{enumerate}
\item[$(A1).$]
$
\sum_{k=1}^n W_{nk}^2/n \stackrel{P}{\rightarrow} 1;
$
\item[$(A2).$]  $\max\limits_{1\le k\le n}|W_{nk}|/\sqrt{n}\stackrel{P}{\rightarrow} 0$.
\end{enumerate}
Then
$$
\frac{1}{\sqrt{n}}\sum_{k=1}^n W_{nk} \xi_k  \stackrel{d}{\longrightarrow}  N(0,1).
$$
\end{thm}

\begin{rem} 
Suppose that  $\mathbb{E}(\xi_i)=0$,~$\mbox{Var}(\xi_i)=1$, and $W_{n1},  \cdots, W_{nn}$ are exchangeable for any $n\ge 1$.
Mason and Newton \cite{M1992}  obtained  the following conditional central limit theorem  under conditions $(A1)$ and $(A2)$:
\be
\sup_{x\in \mathbb{R}}\Big|\mathbb{P}\Big(\frac{1}{\sqrt{n}}\sum_{k=1}^n W_{nk} \xi_k  \le x \Big| \xi_1, \xi_2 \cdots \Big)-\Phi(x)\Big|\rightarrow 0~~a.s.   \label{condCLT}
\ee
This implies $
(\sum_{k=1}^n W_{nk} \xi_k )/\sqrt{n}\stackrel{d}{\rightarrow}  N(0,1).
$
Moreover, by a slight modification of Theorem 2.1 in  \cite{A1996},  the conditional convergence result in (\ref{condCLT}) is equivalent to the unconditional convergence $
(\sum_{k=1}^n W_{nk} \xi_k )/\sqrt{n}\stackrel{d}{\rightarrow}  N(0,1)
$
when $\mbox{Var}(W_{n1})=O(1)$ and $\sum_{k=1}^n W_{nk}=c_0n$ for some fixed constant $c_0$.  However, for general weights $\{W_{nk}\}$,  the conditional central limit theorem  (\ref{condCLT}) may not hold.
\end{rem}

\begin{rem}
The condition $(A1)$ is necessary for the convergence $
(\sum_{k=1}^n W_{nk} \xi_k )/\sqrt{n} \stackrel{d}{\longrightarrow}  N(0,1)
$ to hold for any i.i.d. random variables $\xi_1, \xi_2, \cdots$ with $\mathbb{E}(\xi_1)=0$ and $\mbox{Var}(\xi_1)=1$.
Indeed, if we take $\xi_1, \xi_2, \cdots$ to be i.i.d. $N(0,1)$ random variables and assume that $
(\sum_{k=1}^n W_{nk} \xi_k )/\sqrt{n}\stackrel{d}{\longrightarrow}  N(0,1)
$, then
\bestar
\mathbb{E}\Big(\exp\Big\{it\frac{\sum_{k=1}^n W_{nk} \xi_k }{\sqrt{n}}\Big\}\Big)=\mathbb{E}\Big(\exp\Big\{-\frac{t^2}{2n}\sum_{k=1}^n W_{nk}^2\Big\}\Big)\rightarrow e^{-t^2/2},~~t\in \mathbb{R},
\eestar
which implies  $(A1)$.
\end{rem}

\begin{rem} When the weights $W_{nk}, ~n\ge 1, ~k=1,\cdots, n$ are nonrandom, in order to study the limiting distribution of $(\sum_{k=1}^n W_{nk} \xi_k )/\sqrt{n}$, it is reasonable to
assume that the sequence $\{W_{nk} \xi_k /\sqrt{n}\}$ satisfies the infinite smallness  condition. That is,  for any $\varepsilon>0$,
\bestar
\sup_{1\le k\le n} \mathbb{P} (|W_{nk}\xi_k /\sqrt{n}|\ge \varepsilon)
=\mathbb{P} (|\xi_1|\ge \varepsilon/ M_n)\rightarrow 0,   
\eestar
where $M_n=\max_{1\le k\le n}|W_{nk}|/\sqrt{n}$. If the distribution function of $\xi_i$ has unbounded support,  then the above condition holds if and only if $M_n\rightarrow 0$.
Therefore,  the condition $(A2)$ can be regarded as a version of the infinite smallness  condition.
Below  we give two examples where $(A2)$ fails and  the conclusion of Theorem \ref{th1} does not hold.
\end{rem}

\begin{exam}
For any $n\ge 1$, we set $W_{n1}=\sqrt{n}$ and $W_{nk}=0$ for $2\le k\le n$. Then $(A1)$ holds but $(\sum_{k=1}^n W_{nk} \xi_k )/\sqrt{n}=\xi_1.$
\end{exam}

\begin{exam}\label{exam1.2}
Let $Y_1, Y_2, \cdots$ be i.i.d. non-negative random variables such that $\mathbb{P}(Y_1>x)$ is  slowly varying at $+\infty$. By Theorem 1.2.1 in   \cite{dH1970},  this   is equivalent to
\bestar
\lim_{x\rightarrow \infty}\frac{x^2\mathbb{P}(Y_1>x)}{\mathbb{E} Y_1^2I(Y_1\le x)}
=\infty.~
\eestar
Define $W_{nk}=\sqrt{n}Y_k/\sum_{i=1}^n Y_i$ for $n\ge 1$ and $1\le k\le n$, 
then,  for each $n\ge 1$,  the weights $\{W_{nk}, k=1,\cdots, n\}$ are exchangeable. Moreover, the condition $(A1)$ holds (see Theorem 1.2 in \cite{CG2004}):
$$\frac{1}{n}\sum_{k=1}^n W_{nk}^2=\frac{\sum_{k=1}^n Y_k^2}{(\sum_{k=1}^n Y_k )^2}\stackrel{P}{\rightarrow} 1.$$
 If $\xi_1, \xi_2,\cdots$ are i.i.d.random variables with $\mathbb{E}|\xi_i|<\infty$, then it follows from Corollary 1.11 in  \cite{KM2012}  that
$
(\sum_{k=1}^n W_{nk} \xi_k )/\sqrt{n}\stackrel{d}{\rightarrow} \xi_1.
$
In this case, we have $\max_{1\le k\le n}|W_{nk}|/\sqrt{n}\stackrel{P}{\rightarrow} 1$ by Theorem 3.2 in Darling \cite{D1952}.
\end{exam}

If  $\xi_1$ belongs to the domain of attraction of a normal distribution  with $\mathbb{E}(\xi_1)=0$, then there exists a sequence of positive constants $\{a_n\}$ such that
\be
\frac{1}{a_n}\sum_{k=1}^n \xi_k \stackrel{d}{\longrightarrow} N(0,1). \label{DAN}
\ee
In view of Theorem \ref{th1}, it is natural to investigate whether a corresponding  central limit theorem holds for the randomly weighted sum:
\be
\frac{1}{a_n}\sum_{k=1}^n W_{nk} \xi_k  \stackrel{d}{\longrightarrow} N(0,1). \label{DAN1}
\ee
In this setting, the conditions $(A1)$ and $(A2)$ are not sufficient to guarantee the convergence in  (\ref{DAN1}).

\begin{exam} Let $\xi_i$ be  symmetric with
\bestar
P(|\xi_i|>x)=\left\{
\begin{array}{cc}
x^{-2}, & x\ge 1,\\
1, & x\le 1.
\end{array}
\right.
\eestar
Then (\ref{DAN}) holds with $a_n=\sqrt{n\log n}$ (see Example 3.4.13 in  \cite{D2019}). Let
$\{m(n), n=1,2,\cdots \}\subset \mathbb{N}$ be a sequence  satisfying $m(n)\rightarrow \infty$ and $\log m(n)=o(\log n)$.
For any $n\ge 1$, define 
\bestar
W_{nk}=\left\{
\begin{array}{ll}
\sqrt{n/m(n)},  & 1\le k\le  m(n), \\
0, & m(n)< k\le n. 
\end{array}
\right.
\eestar
Then the conditions $(A1)$ and $(A2)$ are satisfied,  but
\bestar
\frac{\sqrt{\log n}}{\sqrt{\log m(n)}}\frac{\sum_{k=1}^n W_{nk} \xi_k }{a_n}\stackrel{d}{\rightarrow} N(0,1),
\eestar
which implies that (\ref{DAN1}) doesn't hold. 
\end{exam}

In Theorem \ref{th2} below, we establish  a more general result than (\ref{DAN1}).

\begin{thm} \label{th2}
  Suppose that
\begin{enumerate}
\item[$(A3).$]
$
\frac{1}{n}\sum_{k=1}^n W_{nk}^2 \stackrel{d}{\rightarrow} W,
$ where $W$ is some non-negative random variable;
\item[$(A4).$]
$
\lim\limits_{c\rightarrow \infty} \sup\limits_n \frac{1}{n}\sum_{k=1}^n \mathbb{E} (W_{nk}^{2} I(|W_{nk}|>c))=0.
$
\end{enumerate}
If 
$
\sum_{k=1}^n \xi_k  /{a_n}\stackrel{d}{\rightarrow} N(0,1),
$
then
\bestar
\frac{1}{a_n}\sum_{k=1}^n W_{nk} \xi_k  \stackrel{d}{\longrightarrow} \sqrt{W} N,
\eestar
where $N$ is a standard normal random  variable independent of $W$.
\end{thm}

\begin{rem} \label{rem1.5}
The condition $(A4)$ holds if there exists some $\delta>0$ such that
\bestar
\sup \limits_n \frac{1}{n}\sum_{k=1}^n \mathbb{E} (|W_{nk}|^{2+\delta}) <\infty.
\eestar
 More generally,  $(A4)$ is also satisfied if there exists a positive, non-decreasing function $h(x)$ on $\mathbb{R}^+$ such that $h(x)\rightarrow \infty$ as $x\rightarrow \infty$ and $\sup_n \sum_{k=1}^n \mathbb{E}(W_{nk}^2h(|W_{nk}|))/n <\infty$.
The condition $(A4)$   implies $(A2)$ since for any $\varepsilon>0$,
\bestar
\mathbb{P}\Big(\max\limits_{1\le k\le n}|W_{nk}|/\sqrt{n}\ge \varepsilon \Big) \le \sum_{k=1}^n\mathbb{P} (|W_{nk}|\ge \sqrt{n}\varepsilon)
\le\frac{1}{n\varepsilon^2}\sum_{k=1}^n \mathbb{E} (W_{nk}^{2} I(|W_{nk}|>\sqrt{n}\varepsilon)).
\eestar
\end{rem}

Similar to Theorem \ref{th2}, we can also establish the following stable central limit theorem for randomly weighted sums.

\begin{thm} \label{th3}
Assume that
\bestar
\sum_{k=1}^n \xi_k  /{a_n}\stackrel{d}{\rightarrow} S,
\eestar
where $S$ is a stable random variable 
with characteristic function $\mathbb{E}(e^{it S})=e^{-|t|^{\alpha}}$ and $\alpha\in (0,2)$.
Suppose further that
\begin{enumerate}
\item[$(A5).$]
$
\frac{1}{n}\sum_{k=1}^n |W_{nk}|^{\alpha} \stackrel{d}{\rightarrow} F,
$ where $F$ is some distribution function;
 \item[$(A6).$]   There exists $\beta\ge 1$ satisfying $\beta>\alpha$ and  $
\lim\limits_{c\rightarrow \infty} \sup\limits_n \frac{1}{n}\sum_{k=1}^n \mathbb{E} (|W_{nk}|^{\beta} I(|W_{nk}|>c))=0.
$
\end{enumerate}
Then we have
\bestar
\frac{1}{a_n}\sum_{k=1}^n W_{nk} \xi_k \stackrel{d}{\longrightarrow} W^{1/\alpha}  S,
\eestar
where $W$ is a non-negative random variable with distribution function $F$, independent of $S$.
\end{thm}

\section{Proofs of Theorems \ref{th1}-\ref{th3}}\label{sec3}

\begin{proof}[Proof of Theorem \ref{th1}]
For any sequence $\{n_m, ~m=1,2,\cdots \}\subset \mathbb{N}$, by $(A1)$ and $(A2)$ there exists a subsequence $\{n_{m_l}, ~l=1,2,\cdots \}$ so that as $l\rightarrow \infty$,
\bestar
\frac{1}{n_{m_l}}\sum_{k=1}^{n_{m_l}} W_{n_{m_l},k}^2 \rightarrow 1~~a.s.~~~~\mbox{and}~~~~\max\limits_{1\le k\le n_{m_l}}|W_{n_{m_l},k}|/\sqrt{n_{m_l}}\rightarrow 0 ~~a.s.
\eestar
Throughout this section, let $\mathbb{P}_*$ and $\mathbb{E}_*$ denote the conditional probability and  expectation, respectively, given $\{W_{nk}, ~n\ge 1, ~k=1,\cdots, n\}$.
Under $\mathbb{P}_*$,  the array $\{W_{n_{m_l},k}\xi_k /\sqrt{n_{m_l}}, ~k=1,\cdots, n_{m_l}, ~l=1,2,\cdots\}$ satisfies 
the conditions of the classical Lindeberg-Feller theorem (see Theorem 3.4.10 in  \cite{D2019}) almost surely.  Hence
 \bestar
\sup_{x\in \mathbb{R}}\Big| \mathbb{P}_* \Big(\frac{1}{\sqrt{n_{m_l}}}\sum_{k=1}^{n_{m_l}} W_{n_{m_l},k} \xi_k \le x \Big)- \Phi(x)\Big|\rightarrow 0~~\mbox{a.s.~as}~l\rightarrow \infty.
 \eestar
This implies that
 \bestar
\sup_{x\in \mathbb{R}}\Big| \mathbb{P} \Big(\frac{1}{\sqrt{n_{m_l}}}\sum_{k=1}^{n_{m_l}} W_{n_{m_l},k} \xi_k \le x \Big)-\Phi(x)\Big|\rightarrow 0~\mbox{as}~l\rightarrow \infty
 \eestar
and the desired result follows.
\end{proof}

\begin{proof}[Proof of Theorem \ref{th2}]  Define
\bestar
L(x)=\mathbb{E} (\xi_1^2I(|\xi_1|\le x)),~~x\in \mathbb{R}.
\eestar
Then $L(x)$ is a slowly varying function. Moreover (see, e.g., Section XVII.5 of \cite{F1971}),  
for any $\varepsilon>0$, we have
\be
nP(|\xi_1|>\varepsilon a_n)\rightarrow 0,~~~~ \frac{n}{a_n}\mathbb{E}(\xi_1I(|\xi_1|\le \varepsilon a_n))\rightarrow 0, ~~~~\frac{nL(\varepsilon a_n) }{a_n^2} \rightarrow 1.
\label{limitl}
\ee
Let
\bestar
\xi_{nk}=\frac{\xi_k I(|\xi_k |\le  a_n)-\mathbb{E}(\xi_1I(|\xi_1|\le a_n))}{a_n},~~~~~~n\ge 1, ~k=1,\cdots, n,
\eestar
and let $Z_{n1}, \cdots, Z_{nn}$ be i.i.d. $N(0,\sigma_{n1}^2)$ random variables independent of $\{W_{nk}\}$, where
\be
\sigma_{n1}^2=\mbox{Var}(\xi_{n1})=\frac{L(a_n)}{a_n^2}-\frac{(\mathbb{E}(\xi_1I(|\xi_1|\le a_n)))^2}{a_n^2}\sim \frac{1}{n}.  \label{limit2}
\ee

For any $\varepsilon\in (0,1)$, it follows from (\ref{limitl}) that for sufficiently large $n$,
\be
n\mathbb{E} (\xi_{n1}^2I(|\xi_{n1}|> \varepsilon)) &\le&  n\mathbb{E} (\xi_{n1}^2I(|\xi_1|> \varepsilon a_n/2))\nonumber\\
 &\le&  \frac{2n}{a_n^2}\mathbb{E} (\xi_{1}^2I(a_n\ge |\xi_1|>\varepsilon a_n/2))+\frac{2n}{a_n^2}(\mathbb{E}(\xi_1I(|\xi_1|\le a_n)))^2\nonumber\\
 &=&\frac{2nL(\varepsilon a_n/2)}{a_n^2}-\frac{2nL(a_n)}{a_n^2}+o(n^{-1})\rightarrow 0.    \label{limit3}
\ee

Applying the inequality $|e^{ix}-1-ix+x^2/2|\le \min\{x^2, |x|^3/6\}$ for $x\in \mathbb{R}$ (see for instance Lemma 3.3.19 in  \cite{D2019}) gives
\bestar
&&\Big|\mathbb{E}_* (\exp\{it W_{nk} \xi_{nk}\})-1+\frac{t^2}{2}W_{nk}^2 \sigma_{n1}^2\Big|\\
&\le &
\mathbb{E}_* \Big|\exp\{it W_{nk} \xi_{nk}\}-1-itW_{nk}\xi_{nk}+\frac{t^2}{2}W_{nk}^2 \xi_{nk}^2\Big|\\
&\le& C_t \mathbb{E}_* ((W_{nk}^2 \xi_{nk}^2) \wedge (|W_{nk}\xi_{nk}|^3)),
\eestar
where $C_t$ is a positive constant only depending on $t$, and we write  $a\wedge b:=\min\{a,b\}$ for any $a,b\in \mathbb{R}$. Similarly,
\bestar
\Big|\mathbb{E}_* (\exp \{it W_{nk} Z_{nk} \} )-1+\frac{t^2}{2}W_{nk}^2 \sigma_{n1}^2\Big|
\le  C_t \mathbb{E}_* ((W_{nk}^2 Z_{nk}^2) \wedge (|W_{nk}Z_{nk}|^3)).
\eestar
Hence,
\be
&&\Big|\mathbb{E}\Big( \exp\Big\{it \sum_{k=1}^nW_{nk} \xi_{nk}\Big\}\Big)-\mathbb{E}\Big(\exp\Big\{it \sum_{k=1}^n W_{nk} Z_{nk}\Big\}\Big)\Big|\nonumber\\
&\le&\mathbb{E}\Big|\mathbb{E}_*\Big( \exp\Big\{it \sum_{k=1}^nW_{nk} \xi_{nk}\Big\}\Big)-\mathbb{E}_* \Big(\exp\Big\{it \sum_{k=1}^n W_{nk} Z_{nk}\Big\}\Big)\Big|\nonumber\\
&=&\mathbb{E} \Big|\prod_{k=1}^n \mathbb{E}_*  (\exp \{it W_{nk} \xi_{nk} \} )-\prod_{k=1}^n\mathbb{E}_*  (\exp \{it W_{nk} Z_{nk} \} )\Big|\nonumber\\
&\le& \sum_{k=1}^n \mathbb{E} \Big|\mathbb{E}_* ( \exp \{it W_{nk} \xi_{nk} \})-\mathbb{E}_*  (\exp \{it W_{nk} Z_{nk} \} )\Big|\nonumber\\
&\le& C_t\sum_{k=1}^n  \mathbb{E}((W_{nk}^2 \xi_{nk}^2) \wedge (|W_{nk}\xi_{nk}|^3))+ C_t\sum_{k=1}^n  \mathbb{E}((W_{nk}^2 Z_{nk}^2) \wedge (|W_{nk}Z_{nk}|^3)).  \label{diff}
\ee
By applying (\ref{limit2}) and (\ref{limit3}), we have
\bestar
&&\sum_{k=1}^n   \mathbb{E} (W_{nk}^2 \xi_{nk}^2I(|W_{nk}\xi_{nk}|>\varepsilon))\\
&\le & \sum_{k=1}^n  \mathbb{E} (W_{nk}^2 \xi_{nk}^2 I(|W_{nk}|>c) )+\sum_{k=1}^n   \mathbb{E} (W_{nk}^2 \xi_{nk}^2I(|\xi_{nk}|>\varepsilon/c))\\
&\le & n \sigma_{n1}^2 \sup_m \frac{\sum_{k=1}^m  \mathbb{E} (W_{mk}^2 I(|W_{mk}|>c))}{m} + \mathbb{E}(\xi_{n1}^2I(|\xi_{n1}|>\varepsilon/c)) \sum_{k=1}^n   \mathbb{E} (W_{nk}^2) \rightarrow 0
\eestar
as $n\rightarrow \infty$ and then $c\rightarrow \infty$, and
\bestar
\sum_{k=1}^n   \mathbb{E} (|W_{nk}\xi_{nk}|^3 I(|W_{nk}\xi_{nk}|\le \varepsilon))
\le  \varepsilon \sum_{k=1}^n   \mathbb{E} (W_{nk}^2 \xi_{nk}^2)
=\varepsilon \sigma_{n1}^2\sum_{k=1}^n  \mathbb{E} (W_{nk}^2) \rightarrow 0
\eestar
as $n\rightarrow \infty$ and then $\varepsilon\rightarrow 0$, where we have used the fact that $\sup_n \sum_{k=1}^n  \mathbb{E} (W_{nk}^2) /n<\infty$ by $(A4)$.
Hence
\be
&&\sum_{k=1}^n   \mathbb{E} ((W_{nk}^2\xi_{nk}^{2}) \wedge (|W_{nk}\xi_{nk}|^{3}))\nonumber\\
&\le & \sum_{k=1}^n   \mathbb{E} (W_{nk}^2 \xi_{nk}^2I(|W_{nk}\xi_{nk}|>\varepsilon))+ \sum_{k=1}^n   \mathbb{E} (|W_{nk}\xi_{nk}|^3 I(|W_{nk}\xi_{nk}|\le \varepsilon))\rightarrow 0~~~~~~
\label{E23WX}
\ee
as $n\rightarrow \infty$ and then $\varepsilon\rightarrow 0$.
Similarly,
\be
\sum_{k=1}^n   \mathbb{E} ((W_{nk}^2Z_{nk}^{2}) \wedge (|W_{nk}Z_{nk}|^{3}))\rightarrow 0.  \label{E23WZ}
\ee
Note that (by $(A3)$)
\bestar
\mathbb{E}\Big(\exp\Big\{it \sum_{k=1}^n W_{nk} Z_{nk}\Big\}\Big)
&=&\mathbb{E}\Big(\mathbb{E}_*\Big(\exp\Big\{it \sum_{k=1}^n W_{nk} Z_{nk}\Big\}\Big)\Big)\\
&=&\mathbb{E}\Big(\exp\Big\{-\frac{t^2\sigma_{n1}^2}{2}\sum_{k=1}^n W_{nk}^2\Big\} \Big) \\
&\rightarrow& \mathbb{E}(e^{-t^2W/2})=\mathbb{E}(e^{it\sqrt{W}N}).
\eestar
We have
$ \sum_{k=1}^n W_{nk} Z_{nk} \stackrel{d}{\rightarrow} \sqrt{W} N$.
This implies Theorem \ref{th2} by combining (\ref{diff})--(\ref{E23WZ}).
\end{proof}

\begin{proof} [Proof of Theorem \ref{th3}]
Without loss of generality, we may assume that $  \beta\in [1, 2]$ in (A6).
For $\alpha\in (0,2)$,  we can rewrite 
the characteristic exponent of  $\mathbb{E}(e^{itS})=e^{-|t|^{\alpha}}$
as (see for instance \cite{Br1968}, p. 204-206)  
\be
-|t|^{\alpha} =i\tilde{\gamma} t+\int_{-\infty}^{\infty} f(t, x) \nu(dx),
\ee
where $\tilde{\gamma}\in \mathbb{R}$,~ $\nu$ is the L\'{e}vy measure with $\nu(\{0\})=0$, $\nu(dx)=\tilde{c} |x|^{-\alpha-1}dx$ on $\mathbb{R}-\{0\}$,
 $\tilde{c}>0$, and
\be
f(t, x)=e^{itx}-1-itxI(|x|\le 1). \label{ft}
\ee

Define $\xi_{nk}=\xi_k /a_n$ for $1\le k \le n$,
 then $\xi_{n1},\cdots, \xi_{nn}$ are i.i.d. for any $n\ge 1$ and (see Corollary 15.16 in \cite{K2005}) 
\be
&& n \mathbb{P}\circ \xi_{n1}^{-1} \stackrel{v}{\rightarrow} \nu, ~~~~~
n\mathbb{E}(\xi_{n1}I(|\xi_{n1}|\le 1)) \rightarrow  \tilde{\gamma},  \label{stable3}\\
&& n \mathbb{E}(|\xi_{n1}|^{\beta}\wedge 1)\rightarrow \int (|x|^{\beta}\wedge 1)d\nu(x)<\infty,  \label{stable4}
\ee
where $\stackrel{v}{\rightarrow}$ denotes vague convergence.

Suppose that $S, S_1, S_2,\cdots$ are i.i.d. random variables independent of $\{W_{nk}\}$. Write $S_{nk}:=n^{-1/\alpha} S_k $
for $1\le k\le n$.
Then  it follows from the character function of $S$ that
$
\sum_{k=1}^n S_{nk}\stackrel{d}{=}S
$
and
\be
\sum_{k=1}^n W_{nk}S_{nk}\stackrel{d}{=}\Big(\frac{1}{n}\sum_{k=1}^n |W_{nk}|^{\alpha} \Big)^{1/\alpha} S.  \label{stable2}
\ee

By noting that
\bestar
|f(t,x)|\le (1/2)|tx|^{\beta}I(|x|\le 1)+2I(|x|>1),
\eestar
we have
\bestar
&&\Big|\mathbb{E}_* (e^{itW_{nk}\xi_{nk}})-1-it\tilde{\gamma} W_{nk}/n\Big|I(|W_{nk}|>M)\\
&=&\Big|itW_{nk}(\mathbb{E}(\xi_{nk}I(|\xi_{nk}|\le 1))-\tilde{\gamma}/n)+\mathbb{E}_* (f(tW_{nk}, \xi_{nk}))\Big|I(|W_{nk}|>M)\\
&\le&  t\big|\mathbb{E}(\xi_{n1}I(|\xi_{n1}|\le 1))-\tilde{\gamma}/n)\big||W_{nk}|+ C_t
\mathbb{E}(|\xi_{n1}|^{\beta} \wedge 1) |W_{nk}|^{\beta}I(|W_{nk}|>M).
\eestar
Note that $\sup_n \sum_{k=1}^n \mathbb{E}(|W_{nk}|)/n<\infty$ by $(A6)$.
It follows from  (\ref{stable3}), (\ref{stable4}) and $(A6)$ that
\bestar
\limsup_{M\rightarrow \infty}\limsup_{n\rightarrow \infty}\sum_{k=1}^n \mathbb{E}\Big(\Big|\mathbb{E}_* (e^{itW_{nk}\xi_{nk}})-1-it\tilde{\gamma} W_{nk}/n\Big|I(|W_{nk}|>M)\Big)\rightarrow 0.
\eestar
Similarly,
\bestar
\limsup_{M\rightarrow \infty}\limsup_{n\rightarrow \infty}\sum_{k=1}^n \mathbb{E}\Big(\Big|\mathbb{E}_* (e^{itW_{nk}S_{nk}})-1-it\tilde{\gamma} W_{nk}/n\Big|I(|W_{nk}|>M)\Big)\rightarrow 0.
\eestar
 Hence,  we have
\be
\limsup_{M\rightarrow \infty}\limsup_{n\rightarrow \infty}\sum_{k=1}^n \mathbb{E} \Big(\Big|\mathbb{E}_*(e^{itW_{nk}\xi_{nk}})-\mathbb{E}_* (e^{itW_{nk}S_{nk}})\Big|I(|W_{nk}|>M)\Big)=0. \label{stable1}
\ee

Since $(\mathbb{E}e^{it\xi_{n1}})^n\rightarrow e^{-|t|^{\alpha}}$ holds uniformly on any finite interval,  applying Theorem 17.1 in \cite{F1971} yields that
$n (\mathbb{E}e^{it\xi_{n1}}-1)\rightarrow -|t|^{\alpha}$ holds uniformly on any finite interval.  This implies that, for any fixed $t$,  the convergence $n (\mathbb{E}e^{ita\xi_{n1}}-1)\rightarrow -|ta|^{\alpha}$ holds uniformly for $a$ on any finite interval.
Similarly, for any fixed $t$,  $n \mathbb{E}(e^{itaS_{n1}}-1)\rightarrow -|ta|^{\alpha}$ holds uniformly for $a$ on any finite interval.
Hence for any fixed $M$,
\bestar
&&\sum_{k=1}^n \mathbb{E} \Big(\Big|\mathbb{E}_*(e^{itW_{nk}\xi_{nk}})-\mathbb{E}_*(e^{itW_{nk}S_{nk}})\Big|I(|W_{nk}|\le M)\Big)\\
&\le& n \sup_{|a|\le M} \Big|\mathbb{E}(e^{ita\xi_{n1}})-\mathbb{E}(e^{itaZ_{n1}})\Big|\\
&\le&  \sup_{|a|\le M} \Big|n\mathbb{E}(e^{ita\xi_{n1}}-1)+|ta|^{\alpha}\Big|+\sup_{|a|\le M} \Big|n\mathbb{E}(e^{itaS_{n1}}-1)+|ta|^{\alpha}\Big|\rightarrow 0.
\eestar
This together with (\ref{stable1}) implies that
\bestar
&&\Big|\mathbb{E}\Big( \exp\Big\{it  \sum_{k=1}^nW_{nk} \xi_{nk}  \Big\}\Big)-\mathbb{E}\Big(\exp\Big\{it  \sum_{k=1}^nW_{nk} S_{nk}  \Big\}\Big)\Big|\\
&\le&\mathbb{E}\Big|\mathbb{E}_*\Big( \exp\Big\{it \sum_{k=1}^nW_{nk} \xi_{nk}\Big\}\Big)-\mathbb{E}_* \Big(\exp\Big\{it \sum_{k=1}^n W_{nk} S_{nk}\Big\}\Big)\Big|\\
&\le& \sum_{k=1}^n \mathbb{E} \Big|\mathbb{E}_*(e^{itW_{nk}\xi_{nk}})-\mathbb{E}_*(e^{itW_{nk}S_{nk}})\Big|\rightarrow 0
\eestar
and the desired result follows by (\ref{stable2}) and $(A5)$.
\end{proof}

\section{Proof of Theorem  \ref{th1027} } \label{sect4}

In this section, we will prove Theorem  \ref{th1027} by applying the results on randomly weighted sums established in Section \ref{Sect2}.
In order to represent $T_n$ explicitly as a randomly weighted sum of $\{\xi_k, 1\le k\le n\}$, we first introduce a refined definition of $X_n$.

Let $\{\ep_n, n\ge 2\}$ and $\{\eta_n, n\ge 2\}$ be independent sequences of i.i.d. Bernoulli random variables with parameters $p$ and $r$, respectively.
Assume further that these sequences are independent of  $\{\xi_n, n\ge 1\}$ and $\{U_n, n\ge 2\}$.
Define  $X_1=\xi_1$ and  for $n\ge 2$,
\be \label{xndef2}
{X}_n :=
\left\{
\begin{array}{cl}
{X}_{U_n}, & \text{if}~  \ep_n=1,\eta_n=1;\\
 -X_{U_n}, & \text{if}~  \ep_n=1,\eta_n=0; \\
 \xi_{n}, & \text{if}~  \ep_n=0.
\end{array}
\right.
\ee

Now, consider the tree  $\mathbb{T}_n$  with vertex set $\{1,2, \cdots, n\}$ and
edge set $\{(U_k, k): k=2,\cdots, n\}$, which forms a random recursive tree of size $n$.
By using $\{\varepsilon_k, 2\le k\le  n\}$,
we  construct a Bernoulli bond percolation on $\mathbb{T}_n$ with survival parameter $1-p\in [0, 1]$: for each $2\le k\le n$, the edge $(U_k, k)$ is open if $\ep_k=0$ and closed if $\ep_k=1$.

For each vertex $k$ with $\ep_k=0$, we define $W_{nk}=0$.
For those with $\ep_k=1$, we proceed as follows.
Each such  vertex $k$ serves as the root of a connected component in the percolation. We assign to each vertex in this component   a value of $+1$ or $-1$   in increasing order of their labels. Specifically, vertex $k$ is assigned the value $1$. Then, for every other vertex $i$ ($i\ne k$) in the same component,  
we assign it in the same value as its parent $U_i$ if $\eta_i=1$,  and the opposite value if $\eta_i=0$ (note that the parent vertex $U_i$ of any $i\ne k$ in this component also belongs to the same component).
The weight $W_{nk}$ is then defined as the sum of all  assigned values in this component.

By applying the above construction, we obtain the representation
\be
T_n=\sum_{k=1}^n W_{nk}\xi_k. \label{weightwnk}
\ee
Moreover, in the proof of Theorem  \ref{th1027}, we also need to use  some other quantities related to the percolation. For any $n$, if $j$ is the root of a component (i.e., $\ep_j=1$),
let $N_j(n)$ denote the size of the component containing $j$. If $\ep_j=0$, we adopt the convention that $N_j(n)=0$. We define
\bestar
\nu_k(n):=\#\{1\le j\le n: N_j(n)=k\}
\eestar
as the number of components of size $k$.

Before proving Theorem \ref{th1027}, we provide several preliminary lemmas.
 Throughout, we assume the convention $T_0^0=0$.

\begin{lemma} \label{lemma1}
Let $n\ge 1$ be fixed and let $m_1,\cdots, m_n$ be nonnegative integers such that $m_1+\cdots+m_n=n$.
Then, conditional on $(N_1(n),\cdots, N_n(n))=(m_1, \cdots, m_n)$, the weights $\{W_{nj}, 1\le j\le n\}$ 
are independent, and $W_{nj}\stackrel{d}{=}T_{m_j}^0$ for each $j=1,\cdots, n$.  
\end{lemma}

\begin{proof}
Similar to (\ref{xndef2}), we also provide a refined definition of $T_n^0$. 
Set $X_1^0:=1$ and for $n\ge 2$,
\bestar
{X}_n^0:=
\left\{
\begin{array}{ll}
{X}^0_{U_n}, & \text{if}~ \eta_n=1,\\
 -X^0_{U_n}, & \text{if}~ \eta_n=0.
\end{array}
\right.
\eestar
Then
$
T_n^0:=\sum_{k=1}^n X^0_n
$
is an ERW with the initial value  $T_1^0=1$ and the memory parameter $r=\mathbb{P}(
\eta_1=1)$.

By utilizing the constructions of $W_{nk}$ and $T_n^0$, and noting that $\{\eta_n, n\ge 1\}$ are i.i.d. and independent of $\{(U_n,\ep_n),  n\ge 2\}$,  Lemma \ref{lemma1} is obtained  from a direct application of Lemma 5.3 in \cite{B2024}.
\end{proof}

\begin{lemma} \label{lemma0}
For any $\beta\in (0,4]$, we have
\be
\mathbb{E}(|T_n^0|^{\beta})=O(a_r(n))^{\beta/2}),  \label{resultlemma0}
\ee
where
\bestar
a_r(n):=\left\{ 
\begin{array}{ll}
n, &   r<3/4;\\
n\log n, &   r=3/4;\\
n^{4r-2}, &   r>3/4.
\end{array}
\right.
\eestar
\end{lemma}

\begin{proof}
Write $\gamma:=4r-2$.
It follows from the proof of Theorem 3.8 in \cite{BE2017} that
\be 
\mathbb{E}((T_{n+1}^0)^2)=\frac{n+\gamma}{n}\mathbb{E}((T_n^0)^2)+1 \label{rec1}
\ee 
and
\be 
\mathbb{E}((T_{n+1}^0)^4)=1+\frac{6n+2\gamma}{n}\mathbb{E}((T_n^0)^2)+
\frac{n+2\gamma}{n} \mathbb{E}((T_n^0)^4). \label{rec2}
\ee 

Using  (\ref{rec1}) and recalling $T_1^0=1$, we have
\be 
\mathbb{E}((T_n^0)^2)=\frac{\Gamma(n+\gamma)}{\Gamma(n)}\sum_{k=1}^{n} \frac{\Gamma(k)}{\Gamma(k+\gamma)},  \label{ETn}
\ee 
where $\Gamma(\cdot)$ stands for the gamma function.
Note that $\Gamma(n+\gamma)/\Gamma(n)\sim n^{\gamma}$ as $n\rightarrow \infty$.
If $\gamma\le 1$, then 
\bestar
\frac{\Gamma(n+\gamma)}{\Gamma(n)}\sum_{k=1}^{n} \frac{\Gamma(k)}{\Gamma(k+\gamma)}
\sim n^{\gamma} \sum_{k=1}^n k^{-\gamma}\sim  \left\{ 
\begin{array}{ll}
n/(1-\gamma), &   \gamma<1;\\
n\log n, &   \gamma=1.
\end{array}
\right.
\eestar
If $\gamma> 1$, then
\bestar
c_0:=\sum_{k=1}^{\infty} \frac{\Gamma(k)}{\Gamma(k+\gamma)}
<\infty
\eestar
and hence $\mathbb{E}((T_n^0)^2)\sim c_0 n^{\gamma}$. Furthermore,  $c_0=1/((\gamma-1)\Gamma(\gamma))$ by  Lemma B.1 in 
\cite{BE2017}.
Combining the above facts, we obtain that 
\bestar
\mathbb{E}((T_n^0)^2)
\sim \left\{ 
\begin{array}{ll}
n/(3-4r), &   r<3/4;\\
n\log n, &   r=3/4;\\
c_0n^{4r-2}, &   r>3/4.
\end{array}
\right.
\eestar

Similarly, by (\ref{rec2}), there exists a constant $c_1$  depending only on $r$ 
such that
\bestar
\mathbb{E}((T_n^0)^4)=\frac{\Gamma(n+2\gamma)}{\Gamma(n)}\sum_{k=1}^{n} \frac{\Gamma(k)}{\Gamma(k+2\gamma)}d_{k-1}
\sim \left\{ 
\begin{array}{ll}
3n^2/(3-4r), &   r<3/4;\\
3n^2\log^2 n, &   r=3/4;\\
c_1 n^{8r-4}, &   r>3/4,
\end{array}
\right.
\eestar
where $d_0=1$ and
\bestar 
d_n=1+\frac{6n+2\gamma}{n}\mathbb{E}((T_n^0)^2),~~n\ge 1.
\eestar

Hence,  for any $\beta\in (0,4]$, 
\bestar
\mathbb{E}(|T_n^0|^{\beta})\le (\mathbb{E}((T_n^0)^4))^{\beta/4}=O(
(a_r(n))^{\beta/2}).
\eestar
This completes the proof of Lemma \ref{lemma0}.
\end{proof}

For two sequences of positive numbers $(c_n)$ and $(d_n)$, we write $c_n \asymp d_n$
if   $0<\liminf_{n\rightarrow\infty} c_n/d_n\le \limsup_{n\rightarrow\infty} c_n/d_n<\infty$.

\begin{lemma} \label{lemma38}  Define
$
Z_{l}(n)=\sum_{k=1}^n k^l \nu_k(n).
$
For any $0<p<1$ and $l\ge 0$, we have
\be
\mathbb{E}(Z_{l}(n))  \asymp b_{l}(n),   \label{Zln}
\ee
where
\be
b_{l}(n):=\left\{
\begin{array}{ll}
n^{lp}, & lp>1;\\
 n\log n,  & lp=1;\\
 n, & lp<1.
\end{array}
\right.   \label{bln}
\ee
Furthermore,  if $lp<1$, then
\be
\limsup_{m\rightarrow \infty}\limsup_{n\rightarrow \infty} \Big(\frac{1}{n}\sum_{k=m}^n k^l \mathbb{E}(\nu_k(n))\Big) =0.  \label{summn}
\ee
\end{lemma}

\begin{proof} The first step is to prove by induction on $i$ that for any $i \in \mathbb{N}$,  the inequality~$\mathbb{E}(Z_{l}(n)) \le C_l b_{l}(n)$ holds for all $n\in \mathbb{N}$ and all $l\ge 0$ satisfying $i-1\le l<i$. Throughout this proof,  $C_l$ denotes a positive constant   depending only on $p$ and $l$, which  may take a different
value in each appearance.

It is clear for $i=1$ since $\mathbb{E}(Z_{l}(n))\le \mathbb{E}(\sum_{k=1}^n k \nu_k(n))=n$ for any $0\le l<1$.
Assume the result holds for $i=r-1$. We proceed to the case $i=r$.
If $\epsilon_{n+1}=1$ and  $U_{n+1}$ belongs to a connected component of size $k$
 in the percolation at time $n$, then  $\nu_k(n+1)=\nu_k(n)-1, \nu_{k+1}(n+1)=\nu_{k+1}(n)+1$, and $\nu_j(n+1)=\nu_j(n)$ for any $j\not \in \{k, k+1\}$.
Hence, in this case, 
$
 Z_{l}(n+1)-Z_{l}(n)= (k+1)^l-k^l.
$
This implies that
\bestar
\mathbb{E}\Big(Z_{l}(n+1)-Z_{l}(n)\Big|\mathscr{F}_n, \epsilon_{n+1}=1\Big)
=\sum_{k=1}^{n} \frac{k\nu_{k}(n)}{n} ((k+1)^l-k^l),
\eestar
where $\mathscr{F}_n=\sigma(U_2,\cdots, U_n)$.
Note that in the case $\epsilon_{n+1}=0$,  we have $Z_{l}(n+1)-Z_{l}(n)=1$. Therefore,   
\be
\mathbb{E}\Big(Z_{l}(n+1)-Z_{l}(n)\Big|\mathscr{F}_n\Big)=1-p+p\sum_{k=1}^{n} \frac{k\nu_{k}(n)}{n} ((k+1)^l-k^l). \label{condrecur}
\ee
By applying Taylor's theorem,  we obtain, for any $k, l\ge 1$,
\bestar
(k+1)^l-k^l-lk^{l-1}=\frac{l(l-1)}{2}\zeta_k^{l-2}\le l(l-1)2^{|l-2|-1} k^{l-2},
\eestar
 where $k<\zeta_k<k+1$.
It follows from (\ref{condrecur})  that
\bestar
\mathbb{E}(Z_{l}(n+1))-\mathbb{E}(Z_{l}(n))&=&1-p+p\sum_{k=1}^{n} \frac{k \mathbb{E}(\nu_{k}(n))}{n} ((k+1)^l-k^l)\\
&\le & 1-p+lpn^{-1}\mathbb{E}(Z_{l}(n))+C_ln^{-1}\mathbb{E}(Z_{l-1}(n)).
\eestar
By  the induction hypothesis, we have
\be
\mathbb{E}(Z_{l}(n+1)) \le  \frac{n+lp}{n}\mathbb{E}(Z_{l}(n))+C_ln^{-1}b_{l-1}(n), \label{induction}
\ee
where
\bestar
b_{l-1}(n)=
\left\{
\begin{array}{cc}
 n^{(l-1)p}, & (l-1)p>1;\\
 n\log n,  & (l-1)p=1;\\
 n, &  (l-1)p<1.
\end{array}
\right.
\eestar
If $lp=1$, then $b_{l-1}(n)=n$ and
$
\mathbb{E}(Z_{l}(n+1))\le (1+1/n)\mathbb{E}(Z_{l}(n))+C_l.
$
Hence, for $n\ge 2$,
\bestar
\mathbb{E}(Z_{l}(n))\le C_ln\sum_{k=1}^{n} \frac{1}{k}\le C_ln\log n.
\eestar
If $lp\ne 1$ and $(l-1)p<1$, then $b_{l-1}(n)=n$ and by  (\ref{induction}),
\bestar
\mathbb{E}(Z_{l}(n+1))+\frac{C_l(n+1)}{lp-1} \le \frac{n+lp}{n}\Big(\mathbb{E}(Z_{l}(n))+\frac{C_ln}{lp-1} \Big).
\eestar
This implies that
\bestar
\mathbb{E}(Z_{l}(n))+\frac{C_ln}{lp-1} \le \Big(1+\frac{C_l}{lp-1}\Big)a_{n}(l),
\eestar
where
\bestar
a_n(l) := \prod_{k=1}^{n-1} \frac{k+lp}{k} = \frac{\Gamma(n+lp)}{\Gamma(n)\Gamma(lp+1)}=\frac{n^{lp}}{\Gamma(lp+1)} (1+O(n^{-1})).
\label{anl}
\eestar
Hence
\bestar
\mathbb{E}(Z_{l}(n))\le C_ln^{(lp)\vee 1}.
\eestar
If $(l-1)p\ge 1$, then for large $n$,
\begin{eqnarray*}
	b_{l-1}(n+1)-b_{l-1}(n)\le (l-1/2)pn^{-1}b_{l-1}(n).
\end{eqnarray*}
 This together with (\ref{induction}) yields that 
\bestar
\mathbb{E}(Z_{l}(n+1))+\frac{2C_lb_{l-1}(n+1)}{p} \le \frac{n+lp}{n}\Big(\mathbb{E}(Z_{l}(n))+\frac{2C_lb_{l-1}(n)}{p}\Big).
\eestar
Hence,
\bestar
\mathbb{E}(Z_{l}(n))\le C_ln^{lp}.
\eestar
Combining the above facts completes the induction. Therefore,  $\mathbb{E}(Z_{l}(n)) \le C_l b_{l}(n)$ for $l\ge 0$. Similarly, we can establish that  $\mathbb{E}(Z_{l}(n)) \ge C_l b_{l}(n)$ for $l\ge 0$.  Hence (\ref{Zln})
holds.

Next, we  will use a similar approach to prove (\ref{summn}).  Note that  (see \cite{S1955})  for any fixed $k\ge 1$, 
\bestar
\lim_{n\rightarrow \infty}\frac{\mathbb{E}(\nu_k(n))}{n}=\frac{1-p}{p} \mathrm{B}(k, 1+1/p).
\eestar
Hence, for any $0\le l<1$ and $m\in \mathbb{N}$, we get that, as $ n\rightarrow \infty$,
\bestar
\frac{1}{n} \sum_{k=m}^n k^l \mathbb{E}(\nu_k(n)) &\le& \frac{1}{n} \sum_{k=m}^n k \mathbb{E}(\nu_k(n))
= 1-\frac{1}{n} \sum_{k=1}^{m-1} k \mathbb{E}(\nu_k(n))\\
&\rightarrow& 1-\frac{1-p}{p}\sum_{k=1}^{m-1} k\mathrm{B}(k, 1+1/p),
\eestar 
where we have used the fact that $\sum_{k=1}^n k\nu_k(n)=n$.
By applying (8) in \cite{B2024}, we have
\be
\sum_{k=1}^{\infty}\frac{1-p}{p} k\mathrm{B}(k, 1+1/p)=1.  \label{c1p}
\ee
Therefore,
\bestar
\limsup_{m\rightarrow \infty}\limsup_{n\rightarrow \infty}\frac{1}{n} \sum_{k=m}^n k^l \mathbb{E}(\nu_k(n)) \le
\limsup_{m\rightarrow \infty} \frac{1-p}{p}\sum_{k=m}^{\infty} k\mathrm{B}(k, 1+1/p)=0,
\eestar 
which implies (\ref{summn}) for $0\le l<1$.

Suppose  that  (\ref{summn}) holds for  $r\in \mathbb{N}$ and $r-1\le l<r$ satisfying $lp<1$. 
Then for any $r\le l<r+1$ and  $\varepsilon>0$,
by noting that (applying Stirling's formula  gives that $\mathrm{B}(m, 1+1/p)\asymp m^{-1-1/p}$ as $m\rightarrow \infty$)
\bestar
\limsup_{m\rightarrow \infty}\limsup_{n\rightarrow \infty}\frac{m^l \mathbb{E}(\nu_m(n+1))}{n}= \limsup_{m\rightarrow \infty}\frac{1-p}{p} m^l\mathrm{B}(m, 1+1/p)=0
\eestar
holds for any $l$ with $lp<1$,
 there exists $m=m(\varepsilon)$
such that
\bestar
\limsup_{n\rightarrow \infty}\frac{m^l \mathbb{E}(\nu_m(n+1))}{n}<\varepsilon\qquad \mbox{and} \qquad\limsup_{n\rightarrow \infty} \Big(\frac{1}{n}\sum_{k=m}^n k^{l-1} \mathbb{E}(\nu_k(n))\Big)<\varepsilon.
\eestar
Using arguments similar to those in the proof of (\ref{Zln}),
we obtain that  if $lp<1$, then  
for sufficiently large $n$,
\bestar
\sum_{k=m}^{n+1} k^{l} \mathbb{E}(\nu_k(n+1))&\le& \frac{n+lp}{n}\sum_{k=m}^n k^{l} \mathbb{E}(\nu_k(n))+\frac{m^l \mathbb{E}(\nu_{m}(n+1))}{n}+C_ln^{-1}\sum_{k=m}^n k^{l-1} \mathbb{E}(\nu_k(n))\\
&\le& \frac{n+lp}{n}\sum_{k=m}^n k^{l} \mathbb{E}(\nu_k(n))+C_l\varepsilon,
\eestar
and consequently,
\begin{eqnarray*}
	\frac{1}{n}\sum_{k=m}^n k^l \mathbb{E}(\nu_k(n))\le C_l \varepsilon.
\end{eqnarray*}
Hence, (\ref{summn}) holds for $r\le l<r+1$ satisfying $lp<1$. The proof of (\ref{summn}) is complete by induction.
\end{proof}

\begin{lemma} \label{lemmaa1}
Let $\alpha \in (0,2], r\in [0,1]$ and  $p\in (0,1)$ such that $(2r-1)\alpha p<1$. 
Then
\be
\frac{1}{n}\sum_{k=1}^n |W_{nk}|^{\alpha}\stackrel{P}{\longrightarrow} c(\alpha, p,r),  \label{weakfsum}
\ee
where $c(\alpha, p,r)$ is defined in (\ref{calphapr}).
\end{lemma}

\begin{proof} Assume that $\{T_{nk}^0, n,k=1,2,\cdots\}$ is a sequence of independent random variables, indepedent of $\{\nu_{k}(n),  n,k=1,2,\cdots\}$, and
that
$T_{nk}^0\stackrel{d}{=} T_n^0$. By Lemma \ref{lemma1}, we have
\be
\frac{1}{n}\sum_{k=1}^n |W_{nk}|^{\alpha}\stackrel{d}{=}\frac{1}{n}\sum_{k=1}^n \sum_{i=1}^{\nu_k(n)} |T_{ki}^0|^{\alpha}. \label{sumfexp}
\ee
Recall that (see Lemma 3.1 in  \cite{B2024})
\bestar
\frac{\nu_k(n)}{n} \stackrel{P}{\longrightarrow} \frac{1-p}{p} \mathrm{B} (k, 1+1/p).   
\eestar
Consequently, for any fixed $k\ge 1$, we have
\be
\frac{1}{n}\sum_{i=1}^{\nu_k(n)} |T_{ki}^0|^{\alpha}= \frac{\nu_k(n)}{n} \frac{1}{\nu_k(n)}\sum_{i=1}^{\nu_k(n)} |T_{ki}^0|^{\alpha}  \stackrel{P}{\longrightarrow} \frac{1-p}{p}  \mathbb{E}(|T_k^0|^{\alpha})\mathrm{B}(k, 1+1/p).  \label{sumf1}
\ee
Applying Stirling's formula  gives that $\mathrm{B}(k, 1+1/p)\sim C_p k^{-1-1/p}$ as $k\rightarrow \infty$, where $C_p>0$ is a constant depending only on $p$.
Hence, if $(2r-1)\alpha p<1$, then  it follows from Lemma \ref{lemma0} that 
\be
&&\limsup_{m\rightarrow \infty}\limsup_{n\rightarrow \infty}\sum_{k=m+1}^n \mathbb{E}(|T_k^0|^{\alpha})\mathrm{B}(k,1+1/p)\nn\\ 
&\le& C\limsup_{m\rightarrow \infty}\limsup_{n\rightarrow \infty}\sum_{k=m+1}^n (a_r(k))^{\alpha/2}k^{-1-1/p}=0.  \label{sumf2}
\ee
By (\ref{summn}) in Lemma \ref{lemma38}, we get that for any $\varepsilon>0$,
\bestar
&&\limsup_{m\rightarrow \infty}\limsup_{n\rightarrow \infty}\mathbb{P}\Big(\frac{1}{n}\sum_{k=m+1}^n \sum_{i=1}^{\nu_k(n)} |T_{ki}^0|^{\alpha}>\varepsilon \Big)\\
&\le&\limsup_{m\rightarrow \infty}\limsup_{n\rightarrow \infty} \frac{1}{\varepsilon n}\sum_{k=m+1}^n \mathbb{E} \Big( \sum_{i=1}^{\nu_k(n)} |T_{ki}^0|^{\alpha}\Big)\\
&\le& \limsup_{m\rightarrow \infty}\limsup_{n\rightarrow \infty} \frac{1}{\varepsilon n}\sum_{k=m+1}^n \mathbb{E}(|T_k^0|^{\alpha})\mathbb{E} (\nu_k(n))\\
&\le& C\limsup_{m\rightarrow \infty}\limsup_{n\rightarrow \infty} \frac{1}{\varepsilon n}\sum_{k=m+1}^n (a_r(k))^{\alpha/2}\mathbb{E}(\nu_k(n))=0.
\eestar
Then (\ref{weakfsum}) follows by combining this with (\ref{sumfexp}), (\ref{sumf1}) and (\ref{sumf2}).
\end{proof}

We now turn to the proof of Theorem \ref{th1027}.

\begin{proof}[Proof of Theorem \ref{th1027}]
Applying  Lemma \ref{lemmaa1} gives that if $(2r-1)\alpha p<1$, then
\bestar
\frac{1}{n} \sum_{k=1}^n |W_{nk}|^{\alpha} \stackrel{P}{\rightarrow} c(\alpha, p, r).
\eestar

Since  $(2r-1)\alpha p<1$, we can choose $\beta$ such that $\max\{\alpha, 1\}<\beta<4, ~(2r-1)\beta p<1$ and $\beta p<2$.  By applying  Lemma \ref{lemma38} and using
similar arguments as in the proof of Lemma \ref{lemmaa1},  we have 
	\bestar
	\sup \limits_n \frac{1}{n}\sum_{k=1}^n \mathbb{E} (|W_{nk}|^{\beta})&=&
\sup \limits_n \frac{1}{n}\sum_{k=1}^n \mathbb{E}(|T_k^0|^{\beta})\mathbb{E} (\nu_k(n))\\
&\le& C\sup \limits_n \frac{1}{n}\sum_{k=1}^n (a_r(k))^{\beta/2}\mathbb{E} (\nu_k(n)) <\infty.
	\eestar
Hence conditions~(A4) and (A6) hold for $\alpha=2$ and $\alpha\in (0,2)$ respectively.

Now Theorem \ref{th1027} follows immediately from Theorems \ref{th2} and \ref{th3}.
\end{proof}

Finally, we present the proof of (\ref{c2pr}) in Remark \ref{rema1}.

\begin{proof}[Proof of (\ref{c2pr})]
It follows from (\ref{calphapr}), (\ref{ETn}) and  Lemma B.1 in 
\cite{BE2017} that if $(4r-2)p<1$, then 
\bestar
\mathbb{E}((T_n^0)^2)=\frac{\Gamma(n+4r-2)}{\Gamma(n)}\sum_{k=1}^{n} \frac{\Gamma(k)}{\Gamma(k+4r-2)}=\frac{1}{4r-3}\Big(\frac{\Gamma(n+4r-2)}{\Gamma(n)\Gamma(4r-2)}-n\Big),
\eestar
and consequently,
\bestar
c(2,p,r)&=&\frac{1-p}{p}\sum_{n=1}^{\infty} \mathbb{E}(|T_n^0|^{2})\mathrm{B}(n, 1+1/p)  \\
&=&  \frac{1-p}{p(4r-3)}\Big(\sum_{n=1}^{\infty} \frac{\Gamma(n+4r-2)\Gamma(1+1/p)}{\Gamma(4r-2)\Gamma(n+1+1/p)}-\sum_{n=1}^{\infty} \frac{\Gamma(n+1)\Gamma(1+1/p)}{\Gamma(n+1+1/p)} \Big) \\
&=&  \frac{1-p}{p(4r-3)} \Big(\frac{4r-2}{1/p-(4r-2)}-\frac{1}{1/p-1}\Big)\\
&=& \frac{1}{1-(4r-2)p}.
\eestar
The proof of  (\ref{c2pr}) is complete.
\end{proof}

\end{document}